\documentclass[preprint,12pt]{elsarticle}
\usepackage{amsfonts}
\usepackage{amsmath}
\usepackage{amssymb,amsmath,amsthm,amsfonts,eepic,graphicx}

\usepackage{picinpar,subfigure}
\usepackage{fancyhdr}
\usepackage{fullpage}
\usepackage{color}

\usepackage{tabularx}

\newtheorem{theorem}{Theorem}
\newtheorem{proposition}{Proposition}
\newtheorem{lemma}{Lemma}

\newtheorem{claim}{Claim}

\begin{document}
\begin{frontmatter}

\title{Asymptotic  Critical   Radii in Random Geometric Graphs over Higher-dimensional Regions}

\author{Jie Ding} \ead{jieding78@hotmail.com}
\address{School of Computer, Jiangsu University of Science and Technology,  Zhenjiang 212100, China}

 \author{Xiang Wei} %\ead{jieding78@hotmail.com}
\address{School of Computer, Jiangsu University of Science and Technology,  Zhenjiang 212100, China}

\author{Shuai Ma} \ead{mashuai@buaa.edu.cn}
\address{ SKLCCSE Lab,
Beihang University, Beijing 100191, China}

%
%\author{Xinshan Zhu} \ead{xszhu126@126.com}
%\address{
%School of Electrical and Information Engineering, Tianjin University, Tianjin 300072, China}

\begin{abstract}
This article presents  the precise asymptotical distribution of two types of critical transmission radii, defined in terms of
$k-$connectivity and the minimum vertex degree, for a random geometry graph distributed over a  unit-volume region $\Omega\subset \mathbb{R}^d (d\geq 3)$.
\end{abstract}


\end{frontmatter}

\section{Introduction and main results}
\label{sec-intro}

\par Let $\chi_n$ be a uniform $n$-point process over a
  convex region $\Omega\subset\mathbb{R}^d(d\geq2)$, i.e., a set of $n$ independent
points each of which is uniformly distributed over $\Omega$, and every pair of points whose Euclidean distance less than $r_n$ is
connected with an undirected edge. So  a random geometric
graph $G(\chi_n,r_n)$ is obtained.

 $k-$connectivity and
the smallest vertex degree are two interesting topological
properties of a random geometry graph. A graph $G$ is said to be
$k-$connected if there is no set of $k-1$ vertices whose removal
would disconnect the graph. Denote by $\kappa$ the connectivity of
$G$, being the maximum $k$ such that $G$ is $k-$connected.
The minimum vertex degree of $G$ is denoted by $\delta$.   Let
$\rho(\chi_n;\kappa\geq k)$ be the minimum $r_n$ such that
$G(\chi_n,r_n)$ is $k-$connected and $\rho(\chi_n;\delta\geq k)$ be
the minimum $r_n$ such that $G(\chi_n,r_n)$ has the smallest degree $k$, respectively.

When $\Omega$ is a unit-area convex region on $\mathbb{R}^2$, the precise probability distributions of these two types of critical radii have been given in an asymptotic manner:
\begin{theorem}\label{thm:combined-Main}(\cite{Ding-RGG2018, Ding-RGG2025})
Let $\Omega\subset \mathbb{R}^2$ be a unit-area convex region such that the length of
the boundary $\partial\Omega$  is $l$, $k\geq 0$ be an integer and $c>0$ be a constant.
\par (i) If $k>0$, let
     \begin{equation*}\label{eq:Theorem-formula-1}
        r_n=\sqrt{\frac{\log n+(2k-1)\log\log n+\xi}{\pi n }},
     \end{equation*}
where $\xi$ satisfies
$$ \left\{\begin{array}{cc}
     \xi=-2\log \left(\sqrt{e^{-c}+\frac{\pi l^2}{64}}-\frac{l\sqrt{\pi}}{8}\right), & k=1, \\
      \xi=2\log \left(\frac{l\sqrt{\pi}}{2^{k+1}k!}\right)+2c, & k>1. \\
    \end{array}
        \right.
$$
\par (ii) If $k=0$, let
     \begin{equation*}\label{eq:Theorem-formula-2}
        r_n=\sqrt{\frac{\log n+c}{\pi n }}.
     \end{equation*}
Then
    \begin{equation}\label{eq:Theorem-formula-3}
       \lim_{n\rightarrow\infty} \frac n{k!}\int_{\Omega}\left(n|B(x,r_n)\cap\Omega|\right)^k e^{-n|B(x,r_n)\cap\Omega|}dx= e^{-c},
    \end{equation}
and therefore,  the probabilities of the two events
$\rho(\chi_n;\delta\geq k+1)\leq r_n$ and $\rho(\chi_n;\kappa\geq
k+1)\leq r_n$ both converge to $\exp\left(-e^{-c}\right)$  as $n\rightarrow\infty$.
\end{theorem}
%This theorem  firstly reveals how the region shape impacts on the critical transmission ranges,
%  generalising the previous work~\cite{DH-largest-NN, Penrose-RGG-book,Penrose-k-connectivity,PWan04:Asymptotic-critical-transmission-MobileHoc, PJWan-IT-asymptotic-radius} in which only regular regions like disks or squares are considered.
%This paper further demonstrates the asymptotic distribution of the critical radii for convex regions on $\mathbb{R}^3$:

This theorem  firstly reveals how the region shape impacts on the critical transmission ranges,
  generalising the previous work~\cite{DH-largest-NN, Penrose-RGG-book,Penrose-k-connectivity,PWan04:Asymptotic-critical-transmission-MobileHoc, PJWan-IT-asymptotic-radius} in which only regular regions like disks or squares are considered.

  For a unit-volume region $\Omega$ in $\mathbb{R}^3$,  satisfing the following assumption:\\
\textbf{Assumption}~I: $\liminf_{r\rightarrow 0} \inf_{x\in\Omega}\frac{\mathrm{Vol}(B(x,r)\cap \Omega )}{\mathrm{Vol}(B(x,r))}>0,$ \\
the asymptotic distribution of the critical radii are demonstrated in the following Theorem~\ref{thm:Main3D}.

\begin{theorem}(\cite{Ding2022-3DRadii} with modifications)\label{thm:old-3}
Suppose $\Omega\subset \mathbb{R}^3$ is a simply connected compact region with unit-volume,
satisfying the above~\textbf{Assumption}~I.
The boundary $\partial\Omega$ of the region has finite area $\mathrm{Area}(\partial\Omega)$, $k\geq 0$ is an integer and $c>0$ is a constant.
Let
 \begin{equation}\label{eq:Theorem-radius}
 r_n= \left(\frac{\log n+(\frac{3k}{2}-1)\log\log n+\xi}{  \pi n  }\right)^{\frac13},
\end{equation}
where   $\xi$  solves
\begin{equation}\label{eq:Theorem-radius-1}
\mathrm{Area}(\partial\Omega)\pi^{-\frac{1}{3}} e^{-\frac{2\xi}{3}}\left(\frac{2}{3}\right)^k\frac{1}{k!} =e^{-c}.
\end{equation}
Then
    \begin{equation}\label{eq:Theorem-formula-3}
       \lim_{n\rightarrow\infty} \frac n{k!}\int_{\Omega}\left(n|B(x,r_n)\cap\Omega|\right)^k e^{-n|B(x,r_n)\cap\Omega|}dx= e^{-c},
    \end{equation}
and therefore, the probabilities of the two events
$\rho(\chi_n;\delta\geq k+1)\leq r_n$ and $\rho(\chi_n;\kappa\geq
k+1)\leq r_n$ both converge to $\exp\left(-e^{-c}\right)$  as $n\rightarrow\infty$.
\end{theorem}

In general, the \textbf{Assumption}~I does not hold for all regions in $\mathbb{R}^3$. However,
 as long as region $\Omega$ is convex, or the boundary $\partial\Omega$ is $C^2$ smooth, then $\Omega$ satisfies \textbf{Assumption}~I.
Theorem~\ref{thm:Main3D}  improves the results presented in~\cite{Ding2022-3DRadii}.

 This paper further demonstrates the asymptotic distribution of the critical radii for the  regions in $\mathbb{R}^d(d\geq3)$, which satisfy \textbf{Assumption}~I.

\begin{theorem} \label{thm:Main3D}
Suppose $\Omega\subset \mathbb{R}^d (d\geq 3) $  is a simply connected compact region with unit-volume,
satisfying the above~\textbf{Assumption}~I.
The boundary $\partial\Omega$ of the region has finite area $\mathrm{Area}(\partial\Omega)$, $k\geq 0$ is an integer and $c>0$ is a constant.
Let
 \begin{equation}\label{eq:Theorem-radius}
 r_n= \left( \frac{\log n+ \frac{dk-d+1}{d-1} \log\log n+\xi}{ \frac{d}{2(d-1)}V_d(1) n }\right)^{\frac1d},
\end{equation}
where $\xi$ solves
\begin{equation}\label{eq:Theorem-radius-1}
\mathrm{Area}(\partial\Omega)\frac{\left(\frac{d-1}{d}\right)^k\left(\frac{d}{2(d-1)}V_d(1)\right)^{\frac{d-1}{d}}}{e^{\frac{d-1}{d}\xi}{V_{d-1}(1)}k!}  =e^{-c}.
\end{equation}
Then
    \begin{equation}\label{eq:Theorem-formula-3}
       \lim_{n\rightarrow\infty} \frac n{k!}\int_{\Omega}\left(n|B(x,r_n)\cap\Omega|\right)^k e^{-n|B(x,r_n)\cap\Omega|}dx= e^{-c},
    \end{equation}
and therefore, the probabilities of the two events
$\rho(\chi_n;\delta\geq k+1)\leq r_n$ and $\rho(\chi_n;\kappa\geq
k+1)\leq r_n$ both converge to $\exp\left(-e^{-c}\right)$  as $n\rightarrow\infty$.
\end{theorem}

It is clear to see that Theorem~\ref{thm:old-3} is consistent with Theorem~\ref{thm:Main3D} when $d=3$.

\par
We use the following notations throughout this article.
(1) Region $\Omega \subset \mathbb{R}^d$ is a unit-volume    region, and $B(x,r)\subset \mathbb{R}^d$ is a ball
centered at $x$ with  radius $r$.
(2) Notation $|A|$ is a short  for the volume of a measurable set $A\subset \mathbb{R}^d$ and $\|\cdot\|$ represents the length of a line segment. $\mathrm{Area}(\cdot)$ denotes the area of a hyper surface.
(3) $\mathrm{dist}(x, A)=\inf_{y\in A}\|xy\|$ where $x$ is a point and $A$ is a set.
(4) Given any two nonnegative functions $f(n)$ and $g(n)$, if there exist two constants $0<c_1<c_2$ such that
$c_1g(n)\leq f(n)\leq c_2g(n)$ for any sufficiently large $n$, then  denote $f(n)=\Theta(g(n))$.
We also use notations $f(n)=o(g(n))$ and $f(n)\sim g(n)$ to denote that $\lim\limits_{n\rightarrow\infty}\frac{f(n)}{g(n)}=0$ and $\lim\limits_{n\rightarrow\infty}\frac{f(n)}{g(n)}=1$, respectively.  (5). \(V_{d}(r)=\dfrac{\pi^{\tfrac{d}{2}}\,r^{d}}{\Gamma\bigl(\tfrac{d}{2}+1\bigr)}\) denotes the volume of a \(d\)-dimensional ball of radius \(r\); (6). A hyper surface is said to be $C^2$ smooth in this paper, meaning  that its function has continuous second derivatives.

\section{Proof outline of Theorem~\ref{thm:Main3D}}

Follow the proof framework presented in~\cite{Ding-RGG2018, Ding-RGG2025, Ding2022-3DRadii}, to prove Theorem~\ref{thm:Main3D}, it is sufficient to prove the following three propositions.

\begin{proposition}\label{pro:Explicit-Form}
Suppose unit-volume and compact region $\Omega\subset \mathbb{R}^d$
  has finite area   $\mathrm{Area}(\partial\Omega)$,
 $k\geq 0$ is an integer and $c>0$ is a constant.
If $r_n$ is given by (\ref{eq:Theorem-radius})-(\ref{eq:Theorem-radius-1}), then
\begin{equation}\label{eq:Explicit-Form}
       \lim_{n\rightarrow\infty} \frac n{k!}\int_{\Omega}\left(n|B(x,r_n)\cap\Omega|\right)^k e^{-n|B(x,r_n)\cap\Omega|}dx= e^{-c}.
\end{equation}
\end{proposition}
We should point out that this proposition does not need the~\textbf{Assumption}~I condition on $\Omega$.

\begin{proposition}\label{pro:conclusion-1}   Under the assumptions of Theorem~\ref{thm:Main3D},
\begin{equation}\label{eq:conclusion-1}
\lim_{n\rightarrow\infty}\Pr\left\{\rho(\chi_n;\delta\geq k+1)\leq r_n\right\} =e^{-e^{-c}}.
\end{equation}
\end{proposition}

\begin{proposition}\label{pro:conclusion-2}  Under the assumptions of Theorem~\ref{thm:Main3D},

\begin{equation}\label{eq:conclusion-2}
\lim_{n\rightarrow\infty}\Pr\left\{\rho(\chi_n;\delta\geq k+1)=\rho(\chi_n;\kappa\geq
k+1)\right\}=1.
\end{equation}
\end{proposition}

Throughout this article, we define
\begin{equation}\label{eq:psi-function}
       \psi^k_{n,r}(x)= \frac{\left(n|B(x,r)\cap\Omega|\right)^k
       e^{-n|B(x,r)\cap\Omega|}}{k!},
\end{equation}
indicating the probability of that the node at $x$ has $k$ degree, for a Poisson point process

\section{Proof sketch of Proposition~\ref{pro:Explicit-Form}}\label{sec:Pro-1}

\subsection{A important lemma}

For any $t\in[0,r]$, we define
\begin{equation}\label{eq:a(r,t)}
  a(r,t)=|\{x=(x_1,x_2,\cdots,x_d):x_1^2+x_2^2+\cdots+x_d^2\leq r^2, x_1\leq t\}|,
\end{equation}
the volume of a major spherical segment.
 $a(r,t)$  is usually shortly denoted by $a(t)$, if there is   confusion about $r$.
This volume can be expressed as:
$$
a(t)
=V_{d-1}(1)\,r^d\int_{-1}^{\frac{t}{r}}\bigl(1-u^2\bigr)^{\frac{d-1}{2}}\,\mathrm{d}u
$$
where:
\begin{itemize}
%	\item \(V_{d}(r)=\dfrac{\pi^{\tfrac{d}{2}}\,r^{d}}{\Gamma\bigl(\tfrac{d}{2}+1\bigr)}\) denotes the volume of a \(d\)-dimensional ball of radius \(r\);
	\item \(V_{d-1}(1)=\dfrac{\pi^{\tfrac{d-1}{2}}}{\Gamma\bigl(\tfrac{d+1}{2}\bigr)}\) is the volume of the unit \((d-1)\)-dimensional ball;
	\item \(\Gamma(\cdot)\) is the Gamma function.
\end{itemize}

 \begin{lemma}\label{lemma:a(t)}
 Let $r=r_n=\left( \frac{\log n+ \frac{dk-d+1}{d-1} \log\log n+\xi}{ \frac{d}{2(d-1)}V_d(1) n }\right)^{\frac1d}$ and $k\geq 0$, then
$$
    n\int_0^{\frac r2}\frac{(na(t))^ke^{-na(t)}}{k!}dt\sim
   \frac{\left(\frac{d-1}{d}\right)^k\left(\frac{d}{2(d-1)}V_d(1)\right)^{\frac{d-1}{d}}}{e^{\frac{d-1}{d}\xi}{V_{d-1}(1)}k!}, \quad r\rightarrow 0.
$$
\end{lemma}

\begin{proof}
	For convenience, we denote $C_k=\frac{dk-d+1}{d-1}$ and $c_d = \frac{d}{2(d-1)}V_d(1)$. Notice that
	$$
	a'(t)
	=\frac{\mathrm{d}}{\mathrm{d}t}\Bigl[V_{d-1}(1)\,r^d
	\int_{-1}^{t/r}(1-u^2)^{\frac{d-1}2}\,\mathrm{d}u\Bigr]
	\;=\;V_{d-1}(1)\,r^{\,d-1}\Bigl(1-\frac{t^2}{r^2}\Bigr)^{\frac{d-1}2},
	$$
	$$
	a''(t)
	= - (d-1)\,V_{d-1}(1)\; t\; r^{d-3}\left(1 - \frac{t^2}{r^2}\right)^{\frac{d-3}{2}}.
	$$
	
	Let $f(t) = na(t)$, then
	
	\begin{eqnarray*}
		f(0) &=& nV_{d-1}(1)  r^d \int_{-1}^{0} (1 - u^2)^{\frac{d-1}{2}} du \\
		&=& \frac{V_{d-1}(1)\int_{-1}^{0} (1 - u^2)^{\frac{d-1}{2}}du}{c_d} \left( \log n+ C_k \log\log n+\xi\right)\\
		&=& B\left( \log n+ C_k \log\log n+\xi\right), \\
		f\left(\frac r2\right)
		&=&nV_{d-1}(1)  r^d \int_{-1}^{\frac{1}{2}} (1 - u^2)^{\frac{d-1}{2}} du\\
		&=&\frac{V_{d-1}(1)\int_{-1}^{\frac{1}{2}} (1 - u^2)^{\frac{d-1}{2}}du}{c_d} \left( \log n+ C_k \log\log n+\xi\right) \\
		&=& A\left( \log n+ C_k \log\log n+\xi\right), 	\\
		a'(0)
		&=&  V_{d-1}(1)  r^{d-1} \\
		&=& \frac{V_{d-1}(1)}{c_d^{1-\frac{1}{d}}}n^{\frac{1}{d}-1}\left(\log n+ C_k \log\log n+\xi\right)^{1-\frac{1}{d}} \\
		&=& Dn^{\frac{1}{d}-1}\left(\log n+ C_k \log\log n+\xi\right)^{1-\frac{1}{d}},\\
		a'(\frac{r}{2}) &
		=& \left(\frac{3}{4}\right)^{\frac{d-1}{2}}V_{d-1}(1)  r^{d-1} \\
		&=& \left(\frac{3}{4}\right)^{\frac{d-1}{2}}\frac{V_{d-1}(1)}{c_d^{1-\frac{1}{d}}}n^{\frac{1}{d}-1}\left(\log n+ C_k \log\log n+\xi\right)^{1-\frac{1}{d}} \\
		&=& E n^{\frac{1}{d}-1}\left(\log n+ C_k \log\log n+\xi\right)^{1-\frac{1}{d}},
	\end{eqnarray*}

	Here \(A\), \(B\), \(D\), and \(E\) are fixed constants.  Moreover, since
	$$
	(1 - u^2)^{\frac{d-1}{2}}\ge0
	\quad\text{for }u\in[-1,1],
	$$
	the integral defining \(A\) (over \([-1,\tfrac12]\)) is strictly larger than that defining \(B\) (over \([-1,0]\)), and hence $A>B$.

	Starting from
	$$
	B \;=\;\frac{V_{d-1}(1)\,\displaystyle\int_{-1}^{0}(1 - u^2)^{\frac{d-1}{2}}\,du}{c_d},
	$$
	and using the facts
	$$
	V_{d-1}(1)\,\int_{-1}^{0}(1 - u^2)^{\frac{d-1}{2}}\,du
	=\frac{1}{2}\,V_d(1),
	\quad
	c_d=\frac{d}{2(d-1)}\,V_d(1),
	$$
	we immediately obtain $B = \frac{d-1}{d}$.
	
	\begin{equation*}
		\begin{aligned}
			n\int_{0}^{\frac r2} \frac{(f(t))^ke^{-f(t)}}{k!}dt=&n\int_{0}^{\frac r2} \frac{(f(t))^ke^{-f(t)}}{k!} \frac{1}{na'(t)} df(t)\\
			=&   \int_{0}^{\frac r2} \frac{1}{a'(t)}d\left(-e^{-f(t)} \sum_{i=0}^{k}\frac{(f(t))^i}{i!}  \right)\\
			=&-\frac{1}{a'(t)}e^{-f(t)}\sum_{i=0}^{k}\frac{(f(t))^i}{i!}|_{0}^{\frac r2}
			-\int_{0}^{\frac r2} \frac{a''(t)}{(a'(t))^2}e^{-f(t)}\sum_{i=0}^{k}\frac{(f(t))^i}{i!}dt.
		\end{aligned}
	\end{equation*}
	
	(i) First term:

	\begin{equation*}
		\begin{aligned}
			e^{-f(\frac{r}{2})} = n^{-A} (\log n)^{-AC_k}e^{-A\xi},
		\end{aligned}
	\end{equation*}
	
	\begin{equation*}
		\begin{aligned}
			-\frac{1}{a'(\frac{r}{2})}e^{-f(\frac{r}{2})}\sum_{i=0}^{k}\frac{(f(\frac{r}{2}))^i}{i!} =& -\frac{n^{-A} (\log n)^{-AC_k}e^{-A\xi}}{E n^{1-\frac{1}{d}}\left(\log n+ C_k \log\log n+\xi\right)^{1-\frac{1}{d}}} \sum_{i=0}^k\frac{(A\left( \log n+ C_k \log\log n+\xi\right))^i}{i!},
		\end{aligned}
	\end{equation*}
	
	$$
	-\frac{1}{a'(\frac{r}{2})}e^{-f(\frac{r}{2})}\sum_{i=0}^{k}\frac{(f(\frac{r}{2}))^i}{i!} \sim -\frac{1}{a'(\frac{r}{2})}e^{-f(\frac{r}{2})}\frac{(f(\frac{r}{2}))^k}{k!},
	$$
	
	\begin{equation*}
		\begin{aligned}
			-\frac{1}{a'(\frac{r}{2})}e^{-f(\frac{r}{2})}\frac{(f(\frac{r}{2}))^k}{k!} =& -\frac{n^{-A} (\log n)^{-AC_k}e^{-A\xi}}{E n^{\frac{1}{d}-1}\left(\log n+ C_k \log\log n+\xi\right)^{1-\frac{1}{d}}}\frac{(A\left( \log n+ C_k \log\log n+\xi\right))^k}{k!} \\
			=& -\frac{A^k}{e^{A\xi}Ek!} \frac{\left( \log n+ C_k \log\log n+\xi\right)^k}{n^{A+\frac{1}{d}-1}(\log n)^{AC_k}\left(\log n+ C_k \log\log n+\xi\right)^{1-\frac{1}{d}}}.
		\end{aligned}
	\end{equation*}
	
	Since  $A+\frac{1}{d}-1 > B + \frac{1}{d}-1 = 0$, we conclude that
	$$
	 -\frac{1}{a'(\frac{r}{2})}e^{-f(\frac{r}{2})}\frac{(f(\frac{r}{2}))^k}{k!} = o(1),
	$$
	and hence
	$$
		-\frac{1}{a'(\frac{r}{2})}e^{-f(\frac{r}{2})}\sum_{i=0}^{k}\frac{(f(\frac{r}{2}))^i}{i!} = o(1).
	$$
	
	(ii) Second term:
	
	$$
	-\frac{1}{a'(0)}e^{-f(0)}\sum_{i=0}^{k}\frac{(f(0))^i}{i!} \sim -\frac{1}{a'(0)}e^{-f(0)}\frac{(f(0))^k}{k!},
	$$
	
	\begin{equation*}
		\begin{aligned}
			-\frac{1}{a'(0)}e^{-f(0)}\frac{(f(0))^k}{k!} =& -\frac{n^{-B} (\log n)^{-BC_k}e^{-B\xi}}{Dn^{\frac{1}{d}-1}\left(\log n+ C_k \log\log n+\xi\right)^{1-\frac{1}{d}}} \frac{(B\left( \log n+ C_k \log\log n+\xi\right))^k}{k!} \\
			= & -\frac{B^k}{e^{B\xi}Dk!} \frac{\left( \log n+ C_k \log\log n+\xi\right)^k}{n^{B+\frac{1}{d}-1}(\log n)^{BC_k}\left(\log n+ C_k \log\log n+\xi\right)^{1-\frac{1}{d}}}.
		\end{aligned}
	\end{equation*}
	
	Since $B + \frac{1}{d} - 1 = 0$ and $C_k=\frac{dk-d+1}{d-1}$, it follows that
	$$
		BC_k + 1 - \frac{1}{d} = \left(1-\frac{1}{d}\right)\left(\frac{dk-d+1}{d-1}\right)+1-\frac{1}{d} = k.
	$$
	Therefore, the leading contribution of the second term satisfies
	$$
		-\frac{1}{a'(0)}e^{-f(0)}\frac{(f(0))^k}{k!} \sim -\frac{B^k}{e^{B\xi}Dk!}.
	$$

	(iii) Third term

	$$
	a'(t)
	=\;V_{d-1}(1)\,r^{\,d-1}\Bigl(1-\frac{t^2}{r^2}\Bigr)^{\frac{d-1}2},
	$$
	
	$$
	a''(t)
	= - (d-1)\,V_{d-1}(1)\; t\; r^{d-3}\left(1 - \frac{t^2}{r^2}\right)^{\frac{d-3}{2}}.
	$$
	
	When $t\leq \frac{r}{2}$,
	\begin{equation*}
		\begin{aligned}
			\left|\frac{a''(t)}{(a'(t))^3}\right|=&\left|\frac{- (d-1)\,V_{d-1}(1)\; t\; r^{d-3}\left(1 - \frac{t^2}{r^2}\right)^{\frac{d-3}{2}}}{\left(V_{d-1}(1)\,r^{\,d-1}\Bigl(1-\frac{t^2}{r^2}\Bigr)^{\frac{d-1}2}\right)^3}\right| \\
			=& \left|\frac{(d-1)t}{(V_{d-1}(1))^2r^{2d}\left(1 - \frac{t^2}{r^2}\right)^{d}}\right| \\
			\leq& \frac{4^d(d-1)}{2\times 3^d[V_{d-1}(1)]^2r^{2d-1}}.
		\end{aligned}
	\end{equation*}
	
	Then
	\begin{eqnarray*}
		&&\left|\int_{0}^{\frac r2} \frac{a''(t)}{(a'(t))^2}e^{-f(t)}\sum_{i=0}^{k}\frac{(f(t))^i}{i!}dt\right|\\
		&=&\frac{1}{n}\left|\int_{0}^{\frac r2} \frac{a''(t)}{(a'(t))^3}e^{-f(t)}\sum_{i=0}^{k}\frac{(f(t))^i}{i!}df(t)\right|\\
		&\leq & \frac{4^d(d-1)}{2\times 3^d[V_{d-1}(1)]^2} \frac{1}{nr^{2d-1}} \left|\int_{0}^{\frac r2} e^{-f(t)}\sum_{i=0}^{k}\frac{(f(t))^i}{i!}df(t)\right|\\
		&\leq & \Theta(1)\frac{4^d(d-1)}{2\times 3^d[V_{d-1}(1)]^2} \frac{1}{nr^{2d-1}} \left|\int_{0}^{\frac r2} e^{-f(t)}\frac{(f(t))^k}{k!}df(t)\right|\\
		&= & \Theta(1)\frac{4^d(d-1)}{2\times 3^d[V_{d-1}(1)]^2} \frac{1}{nr^{2d-1}} \left|\int_{0}^{\frac r2} d\left(-e^{-f(t)}\sum_{i=0}^{k}\frac{(f(t))^i}{i!}\right) \right|\\
		&\leq & \Theta(1)\frac{4^d(d-1)}{2\times 3^d[V_{d-1}(1)]^2} \frac{1}{nr^{2d-1}}e^{-f(0)}\sum_{i=0}^{k}\frac{(f(0))^i}{i!}.
	\end{eqnarray*}
	
	Due to
	
	$$
	\frac{1}{nr^{2d-1}}e^{-f(0)}\sum_{i=0}^{k}\frac{(f(0))^i}{i!}  \sim
	\frac{1}{nr^{2d-1}}e^{-f(0)}\frac{(f(0))^k}{k!} ,
	$$
	
	\begin{equation*}
		\begin{aligned}
			\frac{1}{nr^{2d-1}}e^{-f(0)}\frac{(f(0))^k}{k!} = & \frac{c_d^{\frac{2d-1}{d}}n^{\frac{d-1}{d}}n^{-B}(\log n)^{-BC_k}e^{-B\xi}}{\left(\log n+ C_k \log\log n+\xi\right)^{\frac{2d-1}{d}}} \frac{(B\left( \log n+ C_k \log\log n+\xi\right))^k}{k!} \\
			= & \frac{c_d^{\frac{2d-1}{d}}B^k}{e^{B\xi}k!}n^{\frac{d-1}{d}-B}\frac{\left( \log n+ C_k \log\log n+\xi\right)^k}{(\log n)^{BC_k}\left(\log n+ C_k \log\log n+\xi\right)^{\frac{2d-1}{d}}}.
		\end{aligned}
	\end{equation*}
	Since $B + \frac{1}{d} - 1 = 0$ and $C_k=\frac{dk-d+1}{d-1}$, it follows that
	$$
	BC_k + 2 - \frac{1}{d} = \left(1-\frac{1}{d}\right)\left(\frac{dk-d+1}{d-1}\right)+2-\frac{1}{d} = k + 1.
	$$
	Therefore, the leading contribution of the third term satisfies
	$$
	\frac{1}{nr^{2d-1}}e^{-f(0)}\frac{(f(0))^k}{k!} = o(1).
	$$

	Therefore,
	$$
	n\int_{0}^{\frac r2} \frac{(f(t))^ke^{-f(t)}}{k!}dt\sim \frac{1}{a'(0)}e^{-f(0)}\frac{(f(0))^k}{k!} \sim \frac{B^k}{e^{B\xi}Dk!}=\frac{\left(\frac{d-1}{d}\right)^k\left(\frac{d}{2(d-1)}V_d(1)\right)^{\frac{d-1}{d}}}{e^{\frac{d-1}{d}\xi}{V_{d-1}(1)}k!}.
	$$	
\end{proof}

In particular, the lemma shows that

(1) $d = 2$, $k\geq 1$:  $V_1(1)=2$, $V_2(1)=\pi$:

$$
\frac{B^k}{e^{B\xi}Dk!} = \frac{\left(\frac{1}{2}\right)^k\left(\frac{2}{2}V_2(1)\right)^{\frac{1}{2}}}{e^{\frac{1}{2}\xi}{V_{1}(1)}k!} = \frac{1}{2^{k+1}} \frac{\pi^{\frac{1}{2}}}{ e^{\frac{1}{2}\xi} k!}
$$

(2) $d = 3, k\geq 0$: $V_2(1)=\pi$, $V_3(1)=\frac{4}{3}\pi$:

$$
\frac{B^k}{e^{B\xi}Dk!} = \frac{\left(\frac{2}{3}\right)^k\left(\frac{3}{4}V_3(1)\right)^{\frac{2}{3}}}{e^{\frac{2}{3}\xi}{V_{2}(1)}k!} = \left(\frac{2}{3}\right)^k\frac{\left(\frac{3}{4}\times \frac{4}{3}\pi\right)^{\frac{2}{3}}}{e^{\frac{2}{3}\xi}\pi k!} = \left(\frac{2}{3}\right)^k \frac{1}{e^{\frac{2}{3}\xi}\pi^{\frac{1}{3}}k!}
$$

\subsection{Proof sketch of Proposition~\ref{pro:Explicit-Form}: $\partial\Omega$ is $C^2$ smooth}

We first prove the result under the assumption that the boundary $\partial \Omega$ is $C^2$ smooth, while later prove the case for general $\Omega$ by approximation. Follow the method shown in~\cite{Ding-RGG2018, Ding-RGG2025, Ding2022-3DRadii},  to prove Proposition~\ref{pro:Explicit-Form}, i.e., $ n\int_{\Omega} \psi^k_{n,r}(x)dx\sim e^{-c}$, we divide $\Omega$ into four disjoint parts (when $r$ is sufficient small):
$$
    \Omega=\Omega(0)\cup \Omega(2) \cup \Omega(1,1) \cup \Omega(1,2),
$$
where
$$
     \Omega(0)=\left\{x\in\Omega:\mathrm{dist}(x,\partial\Omega)\geq r\right\}, \quad
     \Omega(2)=\left\{x\in\Omega:\mathrm{dist}(x,\partial\Omega)\leq
     (G(\Omega)+1)r^2\right\},
$$
and $\Omega(1)= \Omega\backslash (\Omega(0)\cup \Omega(2))$,
$$
    \Omega(1,1)=\left\{x\in \Omega(1):t(x)\leq \frac r2\right\}, \quad
   \Omega(1,2)=\Omega(1)\setminus\Omega(1,1).
$$
The function $t(x)$ is similarly defined in~\cite{Ding-RGG2018, Ding-RGG2025, Ding2022-3DRadii}.
Next, follow the method shown in~\cite{Ding-RGG2018, Ding-RGG2025, Ding2022-3DRadii}, it is not difficult to prove the following four claims:

\begin{claim}\label{pro:Omega(0)}
 $n\int_{\Omega(0)} \psi^k_{n,r}(x)dx= o\left(1\right). $
\end{claim}

 \begin{claim}\label{pro:Omega(2)}
$
   n\int_{\Omega(2)} \psi^k_{n,r}(x)dx=o\left(1\right).
$
\end{claim}

\begin{claim}\label{pro:Omega(1,2)}
$
n\int_{\Omega(1,2)} \psi^k_{n,r}(x)dx=o(1).
$
\end{claim}

\begin{claim}\label{pro:Omega(1,1)}
$$
     n\int_{\Omega(1,1)} \psi^k_{n,r}(x)dx\sim \mathrm{Area}(\partial\Omega)
    n\int_0^{\frac r2}\frac{(na(t))^ke^{-na(t)}}{k!}dt
    \sim
   \mathrm{Area}(\partial\Omega)\frac{\left(\frac{d-1}{d}\right)^k\left(\frac{d}{2(d-1)}V_d(1)\right)^{\frac{d-1}{d}}}{e^{\frac{d-1}{d}\xi}{V_{d-1}(1)}k!}=e^{-c}.
$$
\end{claim}
The four claims prove that $$\int_{\Omega}n\psi^k_{n,r}(x)dx=\left\{\int_{\Omega(0)}+\int_{\Omega(2)}+\int_{\Omega(1,2)}+\int_{\Omega(1,1)}\right\}
    n\psi^k_{n,r}(x)dx\sim e^{-c}.$$

\subsection{Proof   of Proposition~\ref{pro:Explicit-Form}: general $\Omega$}

 For any $\epsilon>0$, there exist  $\Omega_{\epsilon^-} $  and $  \Omega_{\epsilon^+}$ with smooth boundaries, such that $\Omega_{\epsilon^-}\subset \Omega\subset \Omega_{\epsilon^+}$, $|\Omega \backslash \Omega_{\epsilon^-}|+|\Omega_{\epsilon^+} \backslash \Omega_{\epsilon}|<\epsilon$, and
$|\mathrm{Area}(\partial\Omega )-\mathrm{Area}(\partial\Omega_{\epsilon^-})|
+|\mathrm{Area}(\partial\Omega_{\epsilon^+})-\mathrm{Area}(\partial\Omega )|<\epsilon$.
Repeat the previous proof, we similarly obtain
\begin{equation}
  n\int_{\Omega_{\epsilon^-}}\psi^k_{n,r}(x)dx \sim \mathrm{Area}(\partial\Omega_{\epsilon^-})  \frac{\left(\frac{d-1}{d}\right)^k\left(\frac{d}{2(d-1)}V_d(1)\right)^{\frac{d-1}{d}}}{e^{\frac{d-1}{d}\xi}{V_{d-1}(1)}k!},
  \end{equation}
\begin{equation}
  n\int_{\Omega_{\epsilon^+}}\psi^k_{n,r}(x)dx \sim \mathrm{Area}(\partial\Omega_{\epsilon^+})  \frac{\left(\frac{d-1}{d}\right)^k\left(\frac{d}{2(d-1)}V_d(1)\right)^{\frac{d-1}{d}}}{e^{\frac{d-1}{d}\xi}{V_{d-1}(1)}k!}.
\end{equation}
That is, there exists $N>0$ such that $\forall n>N$,
\begin{equation}
 \left|n\int_{\Omega_{\epsilon^-}}\psi^k_{n,r}(x)dx-e^{-c} \right|<M_1\epsilon, \quad \left|n\int_{\Omega_{\epsilon^+}}\psi^k_{n,r}(x)dx-e^{-c} \right|<M_2\epsilon,
\end{equation}
where $e^{-c}= \mathrm{Area}(\partial\Omega)  \frac{\left(\frac{d-1}{d}\right)^k\left(\frac{d}{2(d-1)}V_d(1)\right)^{\frac{d-1}{d}}}{e^{\frac{d-1}{d}\xi}{V_{d-1}(1)}k!}$,
 $M_1$ and $M_2$ are some positive constants.
By
$$
         n\int_{\Omega_{\epsilon^-}}\psi^k_{n,r}(x)dx   \leq n\int_{\Omega }\psi^k_{n,r}(x)dx \leq n\int_{\Omega_{\epsilon^-}}\psi^k_{n,r}(x)dx,
$$
we know
$$
         n\int_{\Omega }\psi^k_{n,r}(x)dx  \sim e^{-c}.
$$

\section{Proof of Proposition~\ref{pro:conclusion-1}}\label{sec:Poissonized-section}

To prove  Proposition~\ref{pro:conclusion-1}, we  instead to prove  its Poissonized version, according to the approach
shown in~\cite{Ding-RGG2018, Ding-RGG2025, Ding2022-3DRadii}.
Then by the standard de-Poissonized technique to see that Proposition~\ref{pro:conclusion-1} holds.

\subsection{Poissonized version of Proposition~\ref{pro:conclusion-1}}

Notice the fact that uniform $n$-point process  $\chi_n$ can be well approximated by a homogeneous Poisson point process $\mathcal{P}_n$ distributed over region $\Omega$ with intensity $n|\Omega|=n$. For any measurable subregion $A\subset\Omega$,
the number of the points in $A$  denoted by $\mathcal{P}_n(A)$, satisfies the following
Poisson distribution with intensity $n|A|$:
$$
   \Pr\left\{\mathcal{P}_n(A)=j\right\}=\frac{e^{-n|A|}(n|A|)^j}{j!},\quad j\geq0.
$$
Given the condition that there are exactly $j$ points in $A$,   these $j$ points are independently and uniformly distributed in $A$.
If $A_i\subset \Omega$ are disjoint, then  $\mathcal{P}_n(A_i)$ are mutually independent Poisson random variables with intensities
$n|A_i|$ respectively, for $i=1,2$. In addition, $\mathcal{P}_n (A_1\cup A_2)=\sum_{i=1}^2\mathcal{P}_n(A_i)$ is a Poisson random variable with intensity $n(|A_1|+|A_2|)$. Such properties make $\chi_n$ can be well approximated by $\mathcal{P}_n$.

Suppose a set of points $\{X_i\}_{i=1}^\infty$ are independent random $d$- vectors having common probability density function $f$, which is almost everywhere continuous on $\mathbb{R}^d$, and the corresponding probability distribution is denoted by $F$. Let
$N_\lambda$ be a Poisson random variable, independent of $\{X_i\}_{i=1}^\infty$, and define Poisson point process
$\mathcal{P}_\lambda=\{X_1,X_2,\cdots,X_{N_\lambda}\}$. $G(\mathcal{P}_\lambda, r)$ is a geometric graph on point set
$\mathcal{P}_\lambda$ with distance parameter $r$, i.e., any pair of two points whose distance is no more than $r$ are connected.
Denote $\mathrm{Po}(\beta)$ for any Poisson random variable with parameter $\beta$.

\begin{lemma}\label{lem:Poisson-approximation}(Theorem~6.7, \cite{Penrose-RGG-book})
Let $W'_{k,\lambda}(r)$ be the number of vertices of degree $k$ in $G(\mathcal{P}_\lambda, r)$.
For $x\in \mathbb{R}^d$, denote $B_x=B(x,r)$ and $\mathcal{P}_\lambda^x=\mathcal{P}_\lambda\cup\{x\}$. Define
\begin{eqnarray*}
 J_1  &=&\lambda^2\int_{R^d}\mathrm{d}F(x)\\
 &\times&\int_{B(x,3r)}\mathrm{d}F(y)\Pr\{\mathcal{P}_\lambda(B_x)=k\}\Pr\{\mathcal{P}_\lambda(B_y)=k\},
\end{eqnarray*}
\begin{eqnarray*}
J_2&=&\lambda^2\int_{R^d}\mathrm{d}F(x)\\
&\times&\int_{B(x,3r)}\mathrm{d}F(y)\Pr\{ \{\mathcal{P}^y_\lambda(B_x)=k\}\cap \{\mathcal{P}^x_\lambda(B_y)=k\}\}.
\end{eqnarray*}
Then
$$
    d_{\mathrm{TV}}(W'_{k,\lambda}(r), \mathrm{Po}(E(W'_{k,\lambda}(r))))\leq 3 (J_1+J_2),
$$
where $d_{\mathrm{TV}}$ refers to total variation distance.
\end{lemma}
This lemma indicates that { $W'_{k,\lambda}(r)$ can be approximated by a Poisson random variable}  in  distribution with parameter
$E(W'_{k,\lambda}(r)))$, as long as $J_1+J_2\rightarrow 0$.

\par For the  uniform $n$-point process  $\chi_n$ over $\Omega$ considered in this article, with uniform density function $f(x)=1$ on $\Omega$, by Palm theory (Theorem~1.6, \cite{Penrose-RGG-book}),
$$
     EW'_{k,n}(r)=  n\int_{\Omega} \psi^k_{n,r}(x)f(x)\mathrm{d}x   = n\int_{\Omega} \psi^k_{n,r}(x)\mathrm{d}x  \sim e^{-c}.
$$
Here $n\int_{\Omega} \psi^k_{n,r}(x)\mathrm{d}x\sim e^{-c}$, which is demonstrated in Proposition~\ref{pro:Explicit-Form},  and has been proved in the above section.

That is to say, random variable $W'_{k,\lambda}(r)$ as  the number of vertices of degree $k$, can satisfy a Poisson distribution $\mathrm{Po}(e^{-c})$ in the asymptotic sense, according to Lemma~\ref{lem:Poisson-approximation}.
Therefore, the proof of Proposition~\ref{pro:conclusion-1}:
$$
 \lim_{n\rightarrow\infty}\Pr\left\{\rho(\chi_n;\delta\geq k+1)\leq r_n\right\} =e^{-e^{-c}},
$$
can be divided into two steps. First, we  prove the the following Proposition~\ref{pro:Poisson-Version}, the
Poissonized version of Proposition~\ref{pro:conclusion-1}.
Then,  by the de-Poissonized technique,   to see the original statement (\ref{eq:conclusion-1}) holds. The de-Poissonized technique is standard and thus omitted here, please see~\cite{Penrose-k-connectivity} for details.

\begin{proposition}\label{pro:Poisson-Version}
Under the assumptions of Theorem~\ref{thm:Main3D},
\begin{equation}\label{eq:concusion-1-Poisson}
\lim_{n\rightarrow\infty}\Pr\left\{\rho(\mathcal{P}_n;\delta\geq k+1)\leq r_n\right\} =\exp\left(-e^{-c}\right),
\end{equation}
where $\mathcal{P}_n$ is a homogeneous Poisson point process of intensity $n$ (i.e., $n|\Omega|$) distributed over unit-volume  region $\Omega$.
\end{proposition}

The proof of Proposition~\ref{pro:Poisson-Version} relies the following  conclusion:
\begin{proposition}\label{proposition1} Under the assumptions of Theorem~\ref{thm:Main3D}, $Z_3$ is a Poisson variable with mean $nv_{y\backslash x}=n|B(y,r)\cap \Omega|-|B(x,r)\cap B(y,r)\cap \Omega|$, then
  \begin{equation*}
\frac{1}{n\pi r^3}\int_{\Omega}n\psi^k_{n,r}(x)dx\int_{\Omega\cap
B(x,r)}n\Pr(Z_3=k-1)dy=o(1).
\end{equation*}
\end{proposition}

Proposition~\ref{proposition1} can be proved by an argument which is similar to the two or three dimension case~\cite{Ding-RGG2018, Ding-RGG2025, Ding2022-3DRadii}, but relies on the following lemma.

\begin{lemma}\label{lem:shadow-low-bound}

The distance $L $ between centers $x$ and $y$ of two hyperspheres, each of radius $r$, satisfies $L < r$. The hyperplane $\gamma$ passing through $x$ and perpendicular to line segment $xy$ determines the semi-hypersphere of the hypersphere centered at $x$ containing $y$, denoted by $B_{\mathrm{semi}}(x,r)$, then the volume of $B_{\mathrm{semi}}(x,r)\setminus B(y,r)$ is
$$V^{*}(L)=|B_{\mathrm{semi}}\setminus B(y,r)|\geq \frac{(d-1)\pi^{\frac{d-1}{2}} L^d}{16\Gamma\left(\frac{d+1}{2}\right)}\left[ 1+ o\left(\left(\frac{L}{r}\right)^2\right) \right].$$
\end{lemma}

\begin{proof} Notice that $ \left( r^2-t^2\right)^{\frac{d-1}{2}}-\left(r^2-(L  -t )^2\right)^{\frac{d-1}{2}}\geq  0$ for   $t\leq L/2<2r$.
\begin{equation*}
\begin{aligned}
V^{*}(L)=&\int_0^{\frac{L  }{2}} \frac{\pi^{\frac{d-1}{2}}}{\Gamma\left(\frac{d+1}{2}\right)} \left[ \left( r^2-t^2\right)^{\frac{d-1}{2}}-\left(r^2-(L  -t )^2\right)^{\frac{d-1}{2}} \right]dt\\
\geq &\frac{\pi^{\frac{d-1}{2}}}{\Gamma\left(\frac{d+1}{2}\right)}\int_0^{\frac{L  }{4}} \left[ \left( r^2-t^2\right)^{\frac{d-1}{2}}-\left(r^2-(L  -t )^2\right)^{\frac{d-1}{2}} \right]dt\\
\geq &\frac{\pi^{\frac{d-1}{2}}}{\Gamma\left(\frac{d+1}{2}\right)}
\int_0^{ \frac{L  }{4}} \left[ \left( r^2- \left(\frac{L}{4}\right)^2 \right)^{\frac{d-1}{2}}-\left(r^2-\left(\frac{3L}{4}\right)^2\right)^{\frac{d-1}{2}} \right]dt\\
= &\frac{L}{4}\frac{\pi^{\frac{d-1}{2}}}{\Gamma\left(\frac{d+1}{2}\right)}
  \left[ \left( r^2-\left(\frac{L}{4}\right)^2\right)^{\frac{d-1}{2}}-\left(r^2-\left(\frac{3L}{4}\right)^2\right)^{\frac{d-1}{2}} \right]\\
=& \frac{Lr^{d-1}}{4}\frac{\pi^{\frac{d-1}{2}}}{\Gamma\left(\frac{d+1}{2}\right)}
  \left[ \left( 1-\frac{1}{16}\frac{L^2}{r^2}\right)^{\frac{d-1}{2}}-\left(1-\frac{9}{16}\frac{L^2}{r^2} \right)^{\frac{d-1}{2}} \right]\\
% \geq & \frac{Lr^{d-1}}{4}\frac{\pi^{\frac{d-1}{2}}}{\Gamma\left(\frac{d+1}{2}\right)}
%  \left[ \left( 1-\frac{1}{16}\frac{L^2}{r^2}\right)^{\frac{d-1}{2}}-\left(1-\frac{9}{16}\frac{L^2}{r^2} \right)^{\frac{d-1}{2}} \right]\\
=& \frac{Lr^{d-1}}{4}\frac{\pi^{\frac{d-1}{2}}}{\Gamma\left(\frac{d+1}{2}\right)}
  \left[1-1+  \frac{8}{16} \frac{d-1}{2}\left(\frac{L}{r}\right)^2 +o\left(\left(\frac{L}{r}\right)^2\right)  \right]\\
=& \frac{(d-1)\pi^{\frac{d-1}{2}} Lr^{d-1} }{16\Gamma\left(\frac{d+1}{2}\right)}    \left(\frac{L}{r}\right)^2 + \frac{(d-1)\pi^{\frac{d-1}{2}} Lr^{d-1} }{16\Gamma\left(\frac{d+1}{2}\right)} o\left(\left(\frac{L}{r}\right)^2\right)   \\
\geq & \frac{(d-1)\pi^{\frac{d-1}{2}} L^d  }{16\Gamma\left(\frac{d+1}{2}\right)}  +\frac{(d-1)\pi^{\frac{d-1}{2}} L^d}{16\Gamma\left(\frac{d+1}{2}\right)} o\left(\left(\frac{L}{r}\right)^2\right) \\
= &\frac{(d-1)\pi^{\frac{d-1}{2}} L^d}{16\Gamma\left(\frac{d+1}{2}\right)}\left[ 1+ o\left(\left(\frac{L}{r}\right)^2\right) \right].
\end{aligned}
\end{equation*}
\end{proof}

\section{Proof of Proposition~\ref{pro:conclusion-2}}

The proof of Proposition~\ref{pro:conclusion-2} is rather simple. Follow the technique demonstrated in~\cite{Penrose-k-connectivity,Ding-RGG2018, Ding-RGG2025, Ding2022-3DRadii}, to construct two events, $E_n(K)$ and $F_n(K)$,  such that for any $K>0$,
$$
     \left\{\rho(\chi_n;\delta\geq k+1)\leq r_n<\rho(\chi_n;\kappa\geq
k+1)\right\}\subseteq E_n(K)\cup F_n(K),
$$
and
$$
\lim_{n\rightarrow \infty}\Pr\{E_n(K)\}=\lim_{n\rightarrow \infty}\Pr\{F_n(K)\}.
$$
The definition of the events and the convergence results are organised in
the following two propositions.
Here and in the following, we do not introduce the notations in detail. Please refer to~\cite{Penrose-k-connectivity}.

\begin{proposition}(Proposition~5.1,\cite{Penrose-k-connectivity})\label{proposition-appendix-1}
Let $E_n(K)$ be the event that there exists a set $U$ of two or more
points of $\chi_n(U_{r_n})$, of diameter at most $Kr_n$ with
$\chi_n(U_{r_n}\backslash U)\leq k$. Then for all $K>0$ we have
$\lim_{n\rightarrow \infty}\Pr\{E_n(K)\}=0$.
\end{proposition}

\begin{proposition}\label{proposition-appendix-2}(Proposition~5.2,\cite{Penrose-k-connectivity})
Let $F_n(K)$ be the event that there are disjoint subsets $U,V,W$ of
$\chi_n$, such that $\mbox{card}(W)\leq k$, and such that $U$ and
$V$ are connected components of the $r_n$ graph on $\chi_n\setminus
W$, with $\mbox{diam}(U)>Kr_n$ and $\mbox{diam}(V)>Kr_n$. Then there
exists $K>0$, such that $\lim_{n\rightarrow \infty}\Pr\{F_n(K)\}=0$.
\end{proposition}

Proposition~\ref{proposition-appendix-1} is immediate from the following three lemmas (\cite{Penrose-k-connectivity}).
\begin{lemma}(Lemma~5.1,\cite{Penrose-k-connectivity})
     Let $K>0$. Then with $\psi^{k}_{n,r}(x)$ defined previously,
     $$
          \Pr\{E_n(K)\}\leq \sup_{x\in\Omega} \left(
          \frac{\Pr\{E_n^x(K)\}}{\psi^{k}_{n,r_n}(x)}
          \right).
     $$
\end{lemma}

\begin{lemma}\label{lemma5}(Lemma~5.2,\cite{Penrose-k-connectivity}) There exists $K\in (0,1]$, such that
   $$
        \lim_{n\rightarrow\infty}\mathop{\sup}_{x\in \Omega}
        \frac{\Pr\left\{E_n^x(K)\right\}}{\psi_{n,r_n}^k(x)}=0.
   $$
\end{lemma}

\begin{lemma} (Lemma~5.3,\cite{Penrose-k-connectivity}) Suppose $0<K'<K<\infty$. Then
   $$
        \lim_{n\rightarrow\infty}\mathop{\sup}_{x\in \Omega}
        \frac{\Pr\left\{E_n^x(K)\backslash E_n^x(K')\right\}}{\psi_{n,r_n}^k(x)}=0.
   $$
\end{lemma}

 The \textbf{Assumption}~1 condition in Theorem~\ref{thm:Main3D} and Proposition~\ref{pro:conclusion-2}, makes that
   $|B(x,r)\cap\Omega|$ in  $
       \psi^k_{n,r}(x)= \frac{\left(n|B(x,r)\cap\Omega|\right)^k
       e^{-n|B(x,r)\cap\Omega|}}{k!}
$
can be bounded by $C|B(x,r)|\leq|B(x,r)\cap\Omega|\leq |B(x,r)|$ where constant $C>0$.

As a result, follow the   approach demonstrated in~\cite{Penrose-k-connectivity,Ding-RGG2018, Ding-RGG2025, Ding2022-3DRadii},  Proposition~\ref{proposition-appendix-1} and~\ref{proposition-appendix-2} holds for general  regions as long as \textbf{Assumption}~1 being satisfied.
Based on these two generalised propositions, a squeezing argument can lead to the following
$$\lim_{n\rightarrow\infty}\Pr\left\{\rho(\chi_n;\delta\geq k+1)=\rho(\chi_n;\kappa\geq
k+1)\right\}=1.$$
That is,  Proposition~\ref{pro:conclusion-2} holds.
Therefore,  the proof of Theorem~\ref{thm:Main3D} is completed.

\section*{References}
\bibliographystyle{elsarticle-num}
%\bibliography{RGG3DFeb5}

\bibliography{RGGHD}

\end{document}